\documentclass[preprint,aap]{imsart}
\pdfoutput=1
\usepackage[T1]{fontenc}
\usepackage[utf8]{inputenc}
\usepackage{amssymb,amsmath,amsthm}
\RequirePackage[colorlinks,citecolor=blue,urlcolor=blue]{hyperref}

\startlocaldefs
\newtheorem{thm}{Theorem}[section]
\newtheorem{defi}[thm]{Definition}

\newcommand{\R}{\mathbb{R}}

\newcommand{\C}{\mathbb{C}}
\renewcommand{\P}{\mathbb{P}}
\newcommand{\E}{\mathbb{E}}
 
\newcommand{\ind}{\text{\usefont{U}{bbold}{m}{n}1}} 
\newcommand{\F}{\mathcal{F}}

\newcommand{\eps}{\varepsilon}
\numberwithin{equation}{section}
\def\journal@id{}
\def\journal@name{}
\def\journal@url{}
\endlocaldefs

\begin{document}

\begin{frontmatter}

\title{Regenerative Properties of the Linear Hawkes Process with Unbounded Memory}
\runtitle{Regenerative Properties of the Hawkes Process}


\author{\fnms{Carl} \snm{Graham}\corref{}\ead[label=e1]{carl.graham@polytechnique.edu}}
\address{CMAP, \'Ecole polytechnique \\ 91128 Palaiseau, France\\ \printead{e1}}
\affiliation{\'Ecole polytechnique, CNRS, IP Paris}

\runauthor{Carl Graham}

\begin{abstract}
We prove regenerative properties for the linear Hawkes process under minimal assumptions on the transfer function,
which may have unbounded support. These results are applicable to sliding window statistical estimators.
We exploit independence in the Poisson cluster point process decomposition, and the regeneration times are not
stopping times for the Hawkes process. 
The regeneration time is interpreted as the renewal time at zero of a M/G/infinity queue,
which yields a formula for its Laplace transform.
When the transfer function admits some exponential moments, we stochastically dominate the cluster length by 
exponential random variables with parameters expressed in terms of these moments. This yields
explicit bounds on the Laplace transform  of the regeneration time in terms of simple integrals or
special functions yielding an explicit negative upper-bound 
on its abscissa of convergence.
These regenerative results allow, \emph{e.g.},  to systematically derive long-time asymptotic results
 in view of statistical applications.
This is illustrated on a concentration inequality previously obtained with coauthors.
 \end{abstract}



\end{frontmatter}

\section{Introduction}

\subsection{Background}
 
Hawkes~\cite{hawkes} introduced the point process on the real line bearing his name in order to model earthquakes. 
Primary shocks arrive with a constant intensity and generate cascades of aftershocks, each shock generating 
direct aftershocks after durations given by an inhomogeneous Poisson process with a fixed intensity function
 in i.i.d.\ fashion. In mathematical terms, the conditional intensity of the point process of the instants 
 of shocks is the sum of the primary shock arrival intensity and of  time-shifts of the intensity function 
 to the instants of past shocks. Thus, the intensity function must be non-negative, and this can only be used to model self-excitation effects.

Hawkes and Oakes~\cite{hawkesoakes} exploited this additive structure to provide a Poisson cluster point process
decomposition of the Hawkes process. It allowed them to prove the existence of a stationary version 
under a sub-criticality assumption on the cascade of aftershocks, 
and has since shown itself to be a powerful tool in the study of Hawkes processes.

Br\'emaud and Massouli\'e~\cite{bremaudmassoulie} generalized this to phenomena in which  
the response to the above sum is modulated by an excitation function. They studied a point process $N$
satisfying an initial condition on $(-\infty,0]$ and with conditional intensity on $(0,\infty)$ given
for a transfer function $h: (0,\infty) \to \R$ and an excitation function $\phi: \R \to \R_+$ by
\begin{equation}
\label{eq-intens-nonlin}
t \in (0,\infty) \mapsto
\phi\biggl(\int_{(-\infty,t)} h(t-s)\,N(\mathrm{d}s)\biggr)
\triangleq \phi\Biggl(\sum_{s \in N, s <t}h(t-s)\Biggr)\,.
\end{equation}
The process in Hawkes~\cite{hawkes} corresponds to 
$h\ge0$ and $\phi (x) = \lambda + \alpha x$ for $x\ge0$ for constants 
$\lambda \ge0$ and $\alpha\ge0$, and this case is called a linear Hawkes process 
and the general case a nonlinear Hawkes process if needed.

If $\phi$ is non-decreasing then the positive values of $h$ can be interpreted as self-excitation and the negative ones 
as self-inhibition, and $\phi$ modulates the response to the superposition of the effects 
of  previous points. This allows to model phenomena featuring both self-excitation and self-inhibition,
which appear in many applicative fields such as seismology~\cite{hawkesadamopoulos,ogata88},
finance~\cite{bacrydelattrehoffmannmuzySPA, bacrydelattrehoffmannmuzy,bacrymuzy2016,Thibault-Rosenbaum2015, Thibault-Rosenbaum2016},
genetics~\cite{reynaudbouretschbath}, or 
neurosciences~\cite{chevalliercaceresdoumicreynaudbouret, ditlevsenlocherbach, reynaudbouretrivoirardtuleaumalot}.
Already in~\cite{bremaudmassoulie} a neuron network was modeled by an interacting system of Hawkes processes,
which provide interesting models such as those 
in~\cite{chevallier,delattrefournierhoffmann,delattrefournier,ditlevsenlocherbach, duartelocherbachost}.

Nonlinear Hawkes processes are much harder to analyze than linear ones, in particular no equivalent of the 
Poisson cluster point process decomposition is known. 
If $\phi$ is non-decreasing and $h$ is non-negative then the natural monotonicity properties 
of the process help,  but considering $h$ taking negative values is much harder.
This is obvious in~\cite{bremaudmassoulie}, 
where the existence and sometimes the uniqueness and attractivity of a stationary version 
of the nonlinear Hawkes process is proved under various sets of assumptions. 

An assumption which appears in many papers of the field and simplifies several proofs in~\cite{bremaudmassoulie}
is that the response function $\phi$ is bounded, another one is that the transfer function $h$ has bounded support 
and thus the Hawkes process has bounded memory. Both are restrictive for applications.

The mathematical analysis of the models and the development, calibration, 
and validation of efficient statistical tools is an important issue.  For instance
Reynaud-Bouret and Roy~\cite{reynaudbouretroy} 
obtained exponential concentration inequalities for ergodic theorems for the linear Hawkes process with bounded
memory, in which necessarily $h\ge0$ (pure self-excitation),
using a coupling similar to what can be found in Berbee \cite{berbee}. 

In case of transfer functions $h$ which may take negative values (self-inhibition), similar results were obtained in 
Costa~\emph{et al.}~\cite{costa2018renewal} by coauthors and myself under the bounded memory assumption
that $h$ has bounded support. The response function was taken to be $\phi:  x\mapsto x^+$ for simplicity of exposition 
but could be generalized.
The bounded memory assumption was crucial in representing the Hawkes process 
by a  Markov process and proving that it has a positive recurrent state.
This yielded long time limit theorems for functional statistical estimators of interest by deriving them
from corresponding limit theorems for i.i.d.\ sequences. 
Notably exponential concentration inequalities were obtained from Bernstein's inequality
using highly intricate renewal computations.

\subsection{Scope of the Paper}

A filtered probability space $(\Omega,\F,(\F_t)_{t\ge0},\P)$ satisfying the usual assumptions and 
supporting all required random elements is given. All processes will be adapted.

The present paper focuses on linear Hawkes process with unbounded memory, 
\emph{i.e.}, with transfer functions with unbounded support,
for which the techniques in~\cite{costa2018renewal} do not apply. 
Even though our main tool is the Poisson cluster point process decomposition,
which generalizes to interacting systems and to marked linear Hawkes processes, we consider the
one-dimensional case so as to focus on the main difficulties and ideas. 
The methods would generalize under adequate assumptions. 

We aim in particular to provide long-time limit results applicable to an important class
of functional statistical estimators involving a sliding window of fixed width~$A$, and introduce
the notion of $A$-regeneration which is natural in this context.
The notation of the explicit dependence on $A$ is often omitted in what follows.

The main contribution of the paper is to exhibit regeneration times for the linear Hawkes process 
and to study them in a precise fashion, under the minimal assumptions $\int h(t)\,\mathrm{d}t  <1$
and $\int t h(t)\,\mathrm{d}t  <\infty$ and a mild condition on the initial condition ensuring that its 
influence eventually vanishes.
The Laplace transform of the regeneration time is expressed in terms of the transfer function $h$.
When $\int \mathrm{e}^{\theta t} h(t)\,\mathrm{d}t  <\infty$ for some $\theta>0$,
the regeneration time is stochastically dominated by random variables with explicit 
Laplace transforms depending on a parameter computable in terms of the Laplace transform of $h$,
which yields explicit exponential moment bounds. 

The main idea is as follows. We construct the Hawkes process using $\F_0$-measurable random 
elements in order to take into account the influence on $(0,\infty)$ of the initial condition,
and a $(\F_t)_{t\ge0}$-Poisson cluster point process 
to add further points on $(0,\infty)$ similarly to the  decomposition of Hawkes and Oakes~\cite{hawkesoakes}.
We then exhibit regeneration times by exploiting the $(\F_t)_{t\ge0}$-strong Markov property of the 
Poisson cluster point process reflecting its  independence properties.
The regeneration times are stopping times for $(\F_t)_{t\ge0}$ but not in general for the Hawkes process. 
Intuitively, the construction allows us to peek ahead of times $t$ in a $\F_t$-measurable way 
and thus to be able to detect such  times at which the past will not  influence the 
future; \cite[p.1570]{bremaudmassoulie} writes
``Note that an unbounded support for the function $h$ makes the memory always infinite, 
and regenerative arguments do not come up naturally, although they may exist'', and we  found one.

Regenerative proprieties open up the toolbox of coupling techniques and of adapting results for 
i.i.d.~sequences by partitioning the process into an initial delay and an independent i.i.d.~sequence of cycles.
In particular the proofs and results in~\cite{costa2018renewal} can be reproduced
here without the bounded memory assumption that the transfer function $h$ has bounded support, 
but here necessarily $h$ is non-negative and covers self-excitation cases only
whereas~\cite{costa2018renewal} allows for $h$ taking negative values and covers self-inhibition cases. 

Good bounds on the regeneration times are essential for applications. 
As in~\cite{costa2018renewal}, we interpret the regeneration time as the renewal time at the empty state
of a $M/G/\infty$ queue with service time given by the sum of the cluster length and  the window width~$A$, 
and use a formula of Tak\`acs \cite{Takacs1956,Takacs1962} expressing the Laplace transform of this renewal time
in terms of the service time c.d.f. 
After this, \cite{costa2018renewal} establishes a general result on exponential tails of $M/G/\infty$ 
renewal times at zero for service times with exponential tails, which they implement for the Hawkes process
using the somewhat loose bound on the cluster length obtained in~\cite[Sect.~1.1]{reynaudbouretroy}
in terms of the exponential moments of $h$.

Here we seek better and more explicit bounds on the cluster length and renewal time.
Surprisingly, studies of super-critical branching random walks
abound, see~\cite{Shi2015,Hu2016,Barral-etal2018,Mallein2018}, \emph{e.g.},
but there is very little recent work on the sub-critical case. 
M\"oller and Rasmussen~\cite{Moller-Rasmussen2005} obtained results based on a formula in 
Hawkes and Oakes~\cite{hawkesoakes} which allow us to stochastically dominate the cluster length 
by an exponential random variable with parameter computable in terms of 
the exponential moments of  $h$.

This yields that the regeneration time is stochastically dominated by the renewal time at zero
of the $M/G/\infty$ queue with service given by the sum of the dominating exponential random variable
and of $A$, which yields explicit computable bounds. Notably, this yields
explicit bounds on the Laplace transform  of the regeneration time in terms of simple integrals or
special functions appearing in like studies for M/M/infinity queues, 
which yield an explicit negative upper-bound 
on its abscissa of convergence.

These regenerative results allow to systematically derive long-time asymptotic results on the Hawkes process,
in particular applicable to sliding window statistical estimators. In this context,
Costa~\emph{et al.}~\cite[Thm~1.5]{costa2018renewal} thus establish 
exponential concentration inequalities using highly intricate computations 
which can be reproduced~\emph{verbatim}
here. We demonstrate the usefulness of the precise explicit bounds we have obtained for the regeneration time
on the simplified version~\cite[Cor.~1.6]{costa2018renewal} of the concentration
inequality and thus provide explicit non-asymptotic exponential bounds. 

\subsection{Notation and Organization}

Section~\ref{sec-cluster-constr} builds from the Poisson cluster point process
the processes which will be used to state and prove regeneration.
Section~\ref{sec-regeneration} defines the notion of $A$-regeneration and then exhibits 
a sequence of regeneration times.
Section~\ref{sec-main-results-times} interprets the regeneration time in terms
of a $M/G/\infty$ queue, uses this to obtain its Laplace transform, 
and  stochastically dominates  it to obtain  explicit bounds.
Section~\ref{sec-applications} discusses applications to limit theorems  for a wide class
of functional statistical estimators, and illustrates the explicit bounds on a concentration inequality. 
Appendix~\ref{appendix-estim}  discusses how to approximate other classes of useful statistical estimators
by the previous class.
Appendix~\ref{appendix-queue}  lists  formulæ on the renewal times and busy periods of
$M/G/\infty$ and $M/M/\infty$ queues. 

When $B$ is a Borel subset of $\R$, we denote by  $\mathcal{N}(B)$ [resp.~$\mathcal{N}_b(B)$]
the space of boundedly finite  [resp.~finite] counting measures, in particular if $B$ is bounded then 
$\mathcal{N}(B) = \mathcal{N}_b(B)$.
There may be abuse of notation 
 in which the restriction $\mu|_B\in \mathcal{N}(B)$ of $\mu \in \mathcal{N}(\R)$ is identified with
 $\mu\ind_B\in \mathcal{N}(\R)$ and reciprocally in which $\mu \in \mathcal{N}(B)$ is identified with
 its extension to a measure on $\R$ using the null measure on $\R\setminus B$.
The shift operator group $(S_t)_{t\in\R}$ satisfies
(with $B-t = \{x-t : x\in B\}$, etc.)
\begin{equation}
\label{shift}
\left\{
\begin{aligned}
&S_t : \mu \in \mathcal{N}(B) \mapsto S_t\mu \in \mathcal{N}(B-t)\,,
\\
&\int_{B-t} f(x)\, S_t \mu(\mathrm{d}x) \triangleq \int_B f(x-t)\, \mu(\mathrm{d}x)\,,
&&f\in \mathcal{B}_b(B-t)\,,
\\
&S_t \mu(C) = \mu (C+t)\,, 
&&C\in\mathcal{N}(B-t)\,.
\end{aligned}
\right.
\end{equation}

Point processes are considered as random variables with values in some specified $\mathcal{N}(B)$,
sometimes with abuse of notation on $\mathcal{N}(\R)$ with null extension outside $B$.
They may be identified with the random set of their atoms since they always are simple, or
with (appropriately started) counting processes in order to use results for classic c\`adl\`ag processes with values in a Polish space.
Daley and Vere-Jones~\cite{daleyverejones2003,daleyverejones2008} is a reference book
on this topic, in which $\mathcal{N}(B)$ [resp.~$\mathcal{N}_b(B)$] are denoted by 
$\mathcal{N}_B^\#$ [resp.~$\mathcal{N}_B$].

\section{Hawkes Process and Poisson Cluster Point Process}
\label{sec-cluster-constr}           

From now on we investigate the linear Hawkes Process satisfying~\eqref{eq-intens-nonlin}
for linear (actually affine) excitation functions $\phi$ and (necessarily) non-negative transfer functions $h$.
We are specifically interested in the case that  $h$ has unbounded support and hence the process 
has unbounded memory.
We often drop the adjective ``linear'' in the following.

\begin{defi}[Hawkes process]
\label{def:Hawkes}
Let a constant $\lambda\ge0$, an integrable function $h: (0,\infty) \to \R_+$, 
and a $\F_0$-measurable boundedly finite point process $N^\mathrm{in}$ on $(-\infty,0]$ be given.
The point process $N$ on $\R$ is a \emph{(linear) Hawkes process} on $(0,\infty)$ with 
\emph{initial condition} $N^\mathrm{in}$, \emph{base intensity} $\lambda$, and \emph{transfer function} $h$ if 
\[
N|_{(-\infty,0]}=N^\mathrm{in}
\]
 and the conditional intensity measure of $N|_{(0,\infty)}$ 
is absolutely continuous w.r.t.\ the Lebesgue measure with density
\begin{equation*}
 t\in (0,\infty) \mapsto \lambda + \int_{(-\infty,t)} h(t-u)\,N(\mathrm{d}u)
\triangleq \lambda + \sum_{s \in N, s <t}h(t-s)\,.
\label{def:Lambda}
\end{equation*}
\end{defi}

Hawkes and Oakes~\cite{hawkesoakes} construct a stationary linear Hawkes process on $\R$ 
as the Poisson cluster point process of the arrival instants  in the following immigration-branching process,
in which all random events are independent:
ancestors arrive on $\R$ as a Poisson process with intensity $\lambda$ (immigration),
and each individual arrived in the system has offspring after durations given by 
an inhomogeneous Poisson process  of intensity function $h$ (branching).
Specifically, an individual arrived at time $s$ has offsprings arriving on $(s,\infty)$ as an inhomogeneous
Poisson process of intensity function $t \mapsto  h (t-s)$.

The number of offspring of an individual  is Poisson of mean $\int h(t)\,\mathrm{d}t<\infty$
and the durations before their births are i.i.d.\ with density $h/ \int h(t)\,\mathrm{d}t$.
Along with  most references, hereafter we make the sub-criticality assumption
\begin{equation}
\label{sub-crit-ass}
 \int h(t)\,\mathrm{d}t  <1\,.
\end{equation}

We adapt this construction for our purposes. Using again the superposition property of Poisson point processes, 
$N|_{(0,\infty)}$ will be constructed as the sum of 
the Hawkes process with initial condition the null (empty) process, and the
point process  of the birth instants in $(0,\infty)$ of offspring of  $N^\mathrm{in}$ 
and of all subsequent descendants of these offspring born in  $(0,\infty)$. 
The latter is a Hawkes process with initial condition $N^\mathrm{in}$ 
and $\lambda =0$.

Consider first the case of null initial condition $N^\mathrm{in}=0$ under the assumption~\eqref{sub-crit-ass}. 
Let $\P^\mathrm{cl}$ be the law on $\mathcal{N}_b(\R_+) $ of the cluster 
generated by an ancestor arriving at time~$0$ including~$0$ itself, and $\Gamma(\mathrm{d}t,\mathrm{d}\mu)$ 
a $(\F_t)_{t\ge0}$-Poisson process on $(0,\infty)\times \mathcal{N}_b(\R_+) $ with intensity measure 
$\lambda \,\mathrm{d}t\otimes \P^\mathrm{cl}(\mathrm{d}\mu)$.
The process $\Gamma$ induces a $(\F_t)_{t\ge0}$-marked Poisson 
process of intensity $\lambda$ with instants $(T_n)_{n\ge1}$ independent of
the marks $(\mu_n)_{n\ge1}$ which are i.i.d.\ of law $\P^\mathrm{cl}$.
Using~\eqref{shift} and identifying measures on Borel subsets of $\R$ with their extension by the 
null measure when necessary, let
\begin{equation*}
\left\{
\begin{aligned}
\psi &: \gamma \in \mathcal{N}((0,\infty)\times \mathcal{N}_b(\R_+) ) 
\mapsto 
\psi(\gamma) = \int S_{-t}\mu\,\gamma(\mathrm{d}t,\mathrm{d}\mu) \in \mathcal{N}(\R) \,,
\\
\psi_t &: \gamma \in \mathcal{N}((0,\infty)\times \mathcal{N}_b(\R_+) ) 
\mapsto 
\psi_t(\gamma) = \psi(\ind_{(0,t]\times\mathcal{N}_b(\R_+)} \gamma ) \in \mathcal{N}(\R) \,,
\,t\ge0\,.
\end{aligned}
\right.
\end{equation*}
Note that
$\psi(\gamma) = \psi_t(\gamma )  +\psi(\ind_{(t,\infty)\times\mathcal{N}_b(\R_+)}  \gamma)$.
As in~\cite{hawkesoakes}, the superposition properties of Poisson point processes imply that   
\begin{equation*}
\psi(\Gamma) \triangleq
\int S_{-t}\mu\,\Gamma(\mathrm{d}t,\mathrm{d}\mu)
\triangleq \sum_{n\ge1}  S_{-T_n} \mu_n
\end{equation*}
is a Hawkes process with initial condition the null point process on $(-\infty,0]$.
For $t\ge0$, 
\[
\psi_t(\Gamma)     
\triangleq \int   S_{-s}\mu \,\ind_{(0,t]}(s)\Gamma(\mathrm{d}s,\mathrm{d}\mu)
\triangleq \sum_{n\ge1} \ind_{\{T_n\le t\}} S_{-T_n} \mu_n
\]
is $\F_t$-measurable and contains all the arrival instants of ancestors arrived up to time $t$
and their descendance, and thus encodes the influence after time $t$  of these ancestors.

Consider now a general initial condition $N^\mathrm{in}=0$, under the assumption~\eqref{sub-crit-ass} 
and the mild natural assumptions (the equality in the second following from Fubini's theorem)
\begin{subequations}
\label{init-cond-ass}
\begin{align}
\label{first-mom-ass}
&\int t h(t)\,\mathrm{d}t  <\infty\,,
\\
&\int_{\{s\le 0<t\}} h(t-s)\,N^\mathrm{in}(\mathrm{d}s)\,\mathrm{d}t 
=
\int_0^\infty h(t) N^\mathrm{in}([-t,0])\,\mathrm{d}t
<\infty\,. 
\end{align}
\end{subequations}
Let $D_0$ be a $\F_0$-measurable version of a Hawkes process satisfying Definition~\ref{def:Hawkes} 
with $\lambda=0$. It is equal to $N^\mathrm{in}$ on $(-\infty,0]$, and on $(0,\infty)$  counts
the arrival instants of the offspring of  $N^\mathrm{in}$, a point of which at $s\le0$ 
generates such offspring as an inhomogeneous Poisson process of intensity function 
$t \mapsto  h(t-s)\ind_{(0,\infty)}(t)$,
and the descendants of these offspring generated according to the usual mechanism, all this
using $\F_0$-measurable random elements.
Since~\eqref{sub-crit-ass} and~\eqref{init-cond-ass} hold, \cite[Thm~1(c), Remarks 5,6]{bremaudmassoulie} 
yields that $D_0$ is well-defined and eventually dies out in total variation.
The superposition properties yield that the point process
\begin{equation*}
N \triangleq D_0 + \psi(\Gamma) 
\triangleq D_0 + \int  S_{-t}\mu\,\Gamma(\mathrm{d}t,\mathrm{d}\mu)
\end{equation*}
is a Hawkes process with initial condition $N^\mathrm{in}$. 
For $t\ge0$, 
\[
D_t  \triangleq D_0 +\psi_t(\Gamma)   
\triangleq 
D_0 + \int  S_{-s}\mu \,\ind_{(0,t]}(s) \Gamma(\mathrm{d}s,\mathrm{d}\mu)
\]
is $\F_t$-measurable and contains $N^\mathrm{in}$ as well as the arrival instants on $(0,\infty)$
of the offspring of $N^\mathrm{in}$ and of ancestors arrived up to time $t$ and the descendants of all these.
It thus encodes the influence after time $t$ of the initial condition and of these ancestors.
Note that $D_t$ is not $\sigma(N|_{(-\infty,t]})$-measurable, but that $N|_{(-\infty,t]} = D_t|_{(-\infty,t]}$.
We collect what we have achieved in the following.

\begin{thm}
\label{thm-construction}
Assume that~\eqref{sub-crit-ass} and~\eqref{init-cond-ass} hold.  The point processes $\psi(\Gamma)$ 
and the $\F_t$-measurable  $\psi_t(\Gamma)$ and $D_t$  for $t\ge0$ are well-defined, and 
\begin{equation}
\label{eq-decomp-t}
N \triangleq D_0 + \psi(\Gamma) = D_t + \psi(\ind_{(t,\infty)\times\mathcal{N}_b(\R_+)}  \Gamma)
\end{equation}
is a Hawkes process with initial condition $N^\mathrm{in}$.
\end{thm}
\begin{proof}
This follows from the above discussion and notably~\cite[Thm~1(c)]{bremaudmassoulie},
and the obvious $\psi(\gamma) = \psi_t(\gamma )  +\psi(\ind_{(t,\infty)\times\mathcal{N}_b(\R_+)}  \gamma)$.
\end{proof}

\section{Regenerative Results for Linear Hawkes Processes}
\label{sec-regeneration}

Our results aim to be applicable to an important class of functional statistical estimators 
involving a sliding window of  length $A$.
For notational convenience the windows are of the form $(t-A,t]$ for $t\ge 0$, 
observation has started by time $-A$,  and at the deterministic time $T>0$ the estimators are of the form
(see \eqref{shift} for the shift operator $S_t$)
\begin{align}
\label{eq-erg-av}
&\frac1{T}\int_0^T f(S_t (N|_{(t-A,t]}))\,\mathrm{d}t
\triangleq
\frac1{T}\int_0^T f((S_t N)|_{(-A,0]}))\,\mathrm{d}t
\,,
\\ \notag
&f\in \mathcal{B}(\mathcal{N}((-A,0]))\,.
\end{align}
To consider observations starting at time $0$ it suffices to shift the process or relabel time
adequately. This useful class is widely studied as in~\cite{reynaudbouretroy,costa2018renewal}.

These estimators are absolutely continuous processes of $T$, and can be used to approximate
in a well-controlled and precise way many estimators based on integrals with respect to $N$
which are pure jump processes of $T$.
For instance, a useful class is constituted of estimators of the form, 
for functions $w$ in $\mathcal{B}(\R)$ with support included in some 
$(-A,0]$,
\begin{align}
\label{eq-erg-jumps} 
&\frac1{T} \int_{(-A,T]\times (-A,T]} w(y-x)\, N(\mathrm{d}x)N(\mathrm{d}y)
\\ \notag
&\quad=\frac1{T} \int_{\{-A < y\le x \le T\}} w(y-x)\, N(\mathrm{d}x)N(\mathrm{d}y)\,.
\end{align}
Appendix~\ref{appendix-estim} shows how to approximate~\eqref{eq-erg-jumps} by~\eqref{eq-erg-av} 
for a well-chosen $f\triangleq f_w$ in a quite controllable way, which
allows to apply results true for~\eqref{eq-erg-av} for $f\triangleq f_w$ 
on~\eqref{eq-erg-jumps}. 
Regenerative techniques may be applied directly to estimators of the form~\eqref{eq-erg-jumps},
but we chose not to develop this.
In a multi-dimensional setting where $N=(N^i : 1\le i\le d)$ this class generalizes under the form
$\bigl(\frac1{T} \int_{\{-A < y\le x \le T\}} w^{ij}(y-x)\, N^i(\mathrm{d}x)N^j(\mathrm{d}y) : 1\le i,j\le d\bigr)$
and provides valuable information on the relations between components.

The following notion corresponds to classical regeneration for the c\`adl\`ag process 
$(S_t (N|_{(t-A,t]}), t\ge0) \triangleq ((S_t N)|_{(-A,0]}), t\ge0)$,
which explains why the past is wide-sense and the future is strict in terms of $N$.

\begin{defi}[$A$-regeneration]
\label{def:A-regeneration}
Let $0\le A <\infty$.
An a.s.\ finite $(\F_t)_{t\ge0}$-stopping time $\rho$ is called an \emph{$A$-regeneration time} if 
\[
\F_\rho
\;\;\text{is independent of}\;\;
S_\rho (N|_{(\rho-A,\infty)} )\triangleq (S_\rho N)|_{(-A,\infty)}
\]
and the latter has same law as the restriction to $(-A,\infty)$ of a Hawkes process with null initial condition,
\emph{i.e.}, as $\psi(\Gamma)|_{(-A,\infty)}$.
Note that $\sigma(N|_{(-\infty,\rho]})\subset \F_\rho$.
\end{defi}

The event that the influence of the initial condition $N^\mathrm{in}$ and of ancestors arrived in $(0,\rho-A]$ 
has vanished in $(\rho-A,\infty)$ and that no ancestor has arrived in $(\rho-A,\rho]$ 
can be conveniently written $D_\rho ((\rho-A,\infty)) =0$.

\begin{thm}
\label{thm-regen}
Assume that~\eqref{sub-crit-ass} and~\eqref{init-cond-ass} hold. Let $0\le A<\infty$.
Then an a.s.~finite $(\F_t)_{t\ge0}$-stopping time $\rho$ is an $A$-regeneration time if and only if 
$D_\rho ((\rho-A,\infty)) =0$.
\end{thm}

\begin{proof}
If $D_\rho ((\rho-A,\infty)) =0$ then, using~\eqref{eq-decomp-t},
\[
S_\rho (N|_{(\rho-A,\infty)}) 
= S_\rho (\psi(\ind_{(\rho,\infty)\times\mathcal{N}_b(\R_+)}  \Gamma)|_{(\rho-A,\infty)}) 
\]
and the strong Markov property of the $(\F_t)_{t\ge0}$-Poisson point process $\Gamma$ yields that
$\rho$ is an $A$-regeneration time. 
The converse is obvious.
\end{proof}

We now state the main regenerative result of the paper.

\begin{thm}
\label{thm-regen-seq}
Assume that~\eqref{sub-crit-ass} and~\eqref{init-cond-ass} hold.Let $0\le A <\infty$ and
\begin{align*}
\tau_0 &\triangleq \tau^A_0 \triangleq \inf\{t\ge 0 : D_t ((t-A,\infty)) =0\}\,,
\\
\tau_k &\triangleq  \tau^A_k \triangleq\inf \{t >\tau^A_{k-1} : D_t ([t-A,\infty)) \neq 0\,,\; D_t ((t-A,\infty)) =0\}\,,
\;k\ge1\,.
\end{align*}
Then $(\tau_k)_{k\ge0}$ is a sequence of $A$-regeneration times.
Thus the delay and the cycles 
\[
N|_{(-A,\tau_0]}\,,
\quad
S_{\tau_{k-1}} (N|_{(\tau_{k-1}-A,\tau_k]} ) \triangleq
(S_{\tau_{k-1}} N)|_{(-A,\tau_k-\tau_{k-1}]}\,,\; k\ge1\,,
\]
are all independent and each cycle $(S_{\tau_{k-1}} N)|_{(-A,\tau_k-\tau_{k-1}]}$ has same law as 
$N|_{(-A,\tau_1]}$ when $N^\mathrm{in}=0$, and thus as $\psi(\Gamma)|_{(-A,\tau]}$ for
\[
\tau \triangleq \tau^A \triangleq
\inf\{t > 0: \psi_t(\Gamma)([t-A,\infty)) \neq 0\,,\; \psi_t(\Gamma)((t-A,\infty))   =0\}\,.
\]
Moreover $\tau$ and the $\tau_k-\tau_{k-1}$ are integrable, and  
the $\tau_k$ are integrable if $\tau_0$ is integrable.
\end{thm}

\begin{proof}
 Theorem~\ref{thm-takacs} below yields that $\tau$ is integrable.
If $N^\mathrm{in}$ is the null point process, 
then $\tau_0=0$ and the result follows from a recursive application of Theorem~\ref{thm-regen}.
For a general initial condition,  $N \triangleq D_0 + \psi(\Gamma)$  in which $D_0$ is an a.s.\ finite point process
and $\psi(\Gamma)$ has a null initial condition. 
(see Section~\ref{sec-cluster-constr}).
Let $U$ denote the last point of $D_0$ and $\tau'_k$  the regeneration times of
$\psi(\Gamma)$. Then $U<\infty$ and $\kappa\triangleq \inf\{k\ge 1 : U\le \tau'_k \} <\infty$
and hence $\tau_0 \le \tau'_{\kappa+1}<\infty$, a.s.
\end{proof}

Simple conditions on $h$ and $N^\mathrm{in}$ imply that $\tau^A_0$ is integrable.
In the important case that $N^\mathrm{in}$ is stochastically dominated by the stationary Hawkes process 
it suffices that additionally $\int t^2 h(t)\,\mathrm{d}t<\infty$,
see Theorem~\ref{thm-delay} below.

In order for these regenerative properties to be useful in applications, tight estimates on the $A$-regeneration times 
are required. This is quite crucial here since in general the $\tau_k$ are not stopping-times for $N$, 
nor even $\sigma(N)$-measurable since the observation of $N$ cannot directly differentiate between ancestors and
descendants. 
Nevertheless, developing statistical methods for $\tau$ is simpler than, \emph{e.g.},
in the case of Nummelin splitting.
In the bounded memory case, if $h$ vanishes on $(A,\infty)$ then the $(\tau_k)_{k\ge0}$  
are the stopping times for $N$ defined just before~\cite[Thm~3.6]{costa2018renewal} 
which can be easily estimated.

\section{Regeneration Time and Stochastic Domination}
\label{sec-main-results-times}

We hereafter assume that the transfer function $h$ satisfies~\eqref{sub-crit-ass} and~\eqref{first-mom-ass}. 
We consider $0\le A<\infty$ and explicitly mark the $A$-dependence for clarity, but it may be dropped
when the context is clear.
We give a series of results on $\tau \triangleq \tau^A$ culminating with a stochastic domination result
providing good bounds in terms of the transfer function $h$. The evaluation of $h$
is well understood and is a primary goal of any statistical study of the Hawkes process. 

The cluster length $L$ is a generic random variable with same law 
as the duration between the arrival instant of 
an ancestor and the last birth instant of his descendants, 
\emph{i.e.}, as the last point of the cluster under the law $\P^\mathrm{cl}$.
Its cumulative distribution function (c.d.f.), with support $[0,\infty)$, is given by
\[
F: x\in\R\mapsto F(x) =\P(L\le x)\,.
\]

A crucial fact is that $\tau \triangleq \tau^A$  is the renewal time at state $0$ 
of a $M/G/\infty$ queue with arrivals with intensity $\lambda$ and generic service duration $L+A$
with c.d.f., with support $[A,\infty)$, given by
\begin{equation}
\label{cdf-A}
F^A: x\in \R \mapsto  F^A(x) = F(x-A)\,.
\end{equation}
As it is well known, see \cite[p.502]{hawkesoakes}, since  
any individual in the branching process has a number of offspring with mean $\int h(t)\,\mathrm{d}t$  
after durations of mean $\int t h(t)\,\mathrm{d}t/ \int h(t)\,\mathrm{d}t$,
\[
\int_0^\infty (1-F^A(u))\,\mathrm{d}u =\E[L] +A
\le
\frac{\int h(t)\,\mathrm{d}t}{1-\int h(t)\,\mathrm{d}t}
\frac{\int t h(t)\,\mathrm{d}t} {\int h(t)\,\mathrm{d}t} +A
<\infty
\]
and much better bounds are available. 
This yields the following.

\begin{thm}[Takács]
\label{thm-takacs}
Assume that~\eqref{sub-crit-ass} and~\eqref{first-mom-ass} hold.
Let $0\le A <\infty$.
The Laplace transform of $\tau^A$ is given by
\begin{align}
\label{A-laplace}
&\E\bigl[\mathrm{e}^{-s \tau^A}\bigr] 
= 1-
\frac1{\lambda + s} 
\biggl(\int_0^\infty \mathrm{e}^{-st - \lambda \int_0^t (1-F^A(u))\,\mathrm{d} u}\,\mathrm{d}t
\biggr)^{-1}\,,
\\ \notag
&s\in \C\,,\;\Re(s) > 0\,.
\end{align}
Moreover  $\E[\tau^A]<\infty$, $\E[(\tau^A)^2]<\infty$ if and only if $\int t^2 h(t)\,\mathrm{d}t  <\infty$, 
and
\begin{align*}
\E[\tau^A] &=  \frac1\lambda \mathrm{e}^{\lambda (\E[L]+A)} <\infty\,,
\\
\E[(\tau^A)^2] 
&= 
\frac2\lambda \mathrm{e}^{2\lambda (\E[L]+A)}
\int_0^\infty
\Bigl(\mathrm{e}^{-\lambda\int_0^t (1-F^A(u))\,\mathrm{d}u}-\mathrm{e}^{-\lambda (\E[L]+A) }\Bigr)\,
\mathrm{d}t 
\\
&\quad+ \frac2{\lambda^2} \mathrm{e}^{\lambda (\E[L]+A)}\,.
\end{align*}
\end{thm}
\begin{proof}
See Takács~\cite[(37)--(39)]{Takacs1956} or \cite[pp.210,211]{Takacs1962}.
\end{proof}

Appendix~\ref{appendix-queue} lists some expressions for~\eqref{A-laplace} and the integral in this formula.
The fact that $\tau^A$ is the sum of an exponential $\mathcal{E}(\lambda)$ idle period 
and an independent busy period $\beta^A$ implies the product form 
$\E\bigl[\mathrm{e}^{-s \tau^A}\bigr] = \frac\lambda{\lambda + s} \E\bigl[\mathrm{e}^{-s \beta^A}\bigr]
$ in \eqref{A-laplace-prod} and the first equality in~\eqref{B-laplace}. 
Theorem~\ref{thm-takacs} can be expressed in terms of $F$ only using \eqref{cdf-A} 
which yields that
\begin{align}
&\label{Atozero1}
\int_0^t (1-F^A(u))\,\mathrm{d}u 
= \ind_{\{t\le A\}} t + \ind_{\{t > A\}}\biggl(A + \int_0^{t-A} (1-F(u))\,\mathrm{d}u \biggr)\,,
\\
\label{Atozero2}
&\int_0^\infty \mathrm{e}^{-st - \lambda \int_0^t(1-F^A(u))\,\mathrm{d} u}\,\mathrm{d}t
\\ \notag
&\quad= \frac{1- \mathrm{e}^{-(\lambda + s)A}}{\lambda + s}
+  \mathrm{e}^{-(\lambda + s)A}
\int_0^\infty \mathrm{e}^{-st - \lambda\int_0^t(1-F(u))\,\mathrm{d} u}\,\mathrm{d}t\,.
\end{align}
This yields relations between $\tau^A$ and $\tau^0$ such as~\eqref{lapl-A-zero-tau},
$\beta^A$ and $\beta^0$ such as~\eqref{lapl-A-zero-beta}, 
and formulæ for $\beta^A$ and hence $\tau^A$ such as~\eqref{B-laplace-Atozero}.

There is an apparent analytic singularity in the Laplace transform~\eqref{A-laplace} at $s=0$.
It was removed under exponential decay assumptions 
on $1-F^A$ in~\cite[Thm A.1]{costa2018renewal} using integration by parts, 
analytic continuation and Laplace transform properties to provide a negative upper-bound 
for the abscissa of convergence.
Here we will go further and use explicit computations.
For $\theta >0$ we denote the c.d.f.\ of the exponential law $\mathcal{E}(\theta)$ by
\[
G^\theta: x\in\R \mapsto G^\theta(x) = (1-\mathrm{e}^{-\theta x})^+\,.
\]

\begin{thm}[M\"oller and Rasmussen]
\label{thm-sto-dom-cl}
Assume that $\int h(t)\,\mathrm{d}t<1$. 
If $\theta >0$ is such that $\int \mathrm{e}^{\theta t}h(t)\,\mathrm{d}t\le 1$
then $G^\theta\le F$. Then $L$ is stochastically dominated by the exponential random variable $\mathcal{E}(\theta)$.
Notably $\E[L]\le1/\theta$.
\end{thm}
\begin{proof}
Follows from
M\"oller and Rasmussen~\cite[Prop.~3, Lemma~2]{Moller-Rasmussen2005}.
\end{proof}

If $\int \mathrm{e}^{\theta t}h(t)\,\mathrm{d}t < \infty$ for some $\theta> 0$ then,
using $\int h(t)\,\mathrm{d}t < 1$ and monotone convergence,
\begin{equation}
\label{theta-star}
\left\{
\begin{aligned}
&\theta^* \triangleq \sup\biggl\{\theta >0 : \int \mathrm{e}^{\theta t}h(t)\,\mathrm{d}t \le 1\biggr\}>0\,,
\\
&\forall \theta \in (0,\theta^*)\,,\; \int \mathrm{e}^{\theta t}h(t)\,\mathrm{d}t\le 1\,,
\end{aligned}
\right.
\end{equation}
and
 if $1\le \int \mathrm{e}^{\theta t}h(t)\,\mathrm{d}t < \infty$ for some $\theta> 0$ then
$\int \mathrm{e}^{\theta^* t}h(t)\,\mathrm{d}t= 1$.

Theorem~\ref{thm-sto-dom-cl} yields much tighter bounds on $L$ than~\cite[Sect.~1.1]{reynaudbouretroy},
in which $p = \int h(t)\,\mathrm{d}t$, $W_\infty$ denotes the cardinal of the cluster, $L$ is denoted by $H$ 
and bounded by a sum of $W_\infty-1$ random variables of density $h/p$ (all independent), 
and it is assumed that $\ell(\theta) \triangleq \log(\int \mathrm{e}^{\theta t}h(t)\,\mathrm{d}t/ p) \le p - \log p -1$
to conclude that $\P(L>x ) \le \mathrm{e}^{1-p} \mathrm{e}^{-\theta x}$.
In this notation, Theorem~\ref{thm-sto-dom-cl} assumes  that $\ell(\theta) \le - \log p$
to conclude that $\P(L>x ) \le \mathrm{e}^{-\theta x}$. 
Since  $- \log p =  1-p + (p - \log p -1)  > p - \log p -1  = (1-p)^2/2 + o_{p\to1}((1-p)^2)$,
Theorem~\ref{thm-sto-dom-cl} provides much better exponents when $p$ is close to $1$, and 
the absence of a prefactor $\mathrm{e}^{1-p}>1$ yields stochastic domination. 

For $0\le A<\infty$ and $\theta >0$, let $\tau^{\theta,A}$ denote the renewal time at  $0$ 
and $\beta^{\theta,A}$ the busy period of a $M/G/\infty$ 
queue with arrivals at rate $\lambda$ and service time given by the sum of 
an exponential random variable $\mathcal{E}(\theta)$ and of $A$ and hence
with c.d.f. $G^{\theta,A} : x\mapsto(1-\mathrm{e}^{-\theta (x-A)})^+$.
We refer to Appendix~\ref{appendix-queue}  for various expressions for their Laplace Transforms. 

\begin{thm}
\label{thm-sto-dom-reg}
Assume that $\int h(t)\,\mathrm{d}t<1$. Let $0\le A<\infty$.
If $\theta >0$ is such that $\int \mathrm{e}^{\theta t}h(t)\,\mathrm{d}t\le 1$
then $\tau^A$ is stochastically dominated by $\tau^{\theta,A}$. 
Theorem~\ref{thm-takacs} and \eqref{cdf-A}-\eqref{Atozero1}-\eqref{Atozero2} remain true with 
$\tau^A$, $F^A$, $F$, $\E[L]$ replaced by
$\tau^{\theta,A}$,  $G^{\theta,A}$, $G^\theta$, $1/\theta$, and the formulæ are made explicit by
\begin{align}
\label{int-mminfty}
\int_0^\infty \mathrm{e}^{-st - \lambda\int_0^t(1-G^\theta(u))\,\mathrm{d} u}\,\mathrm{d}t
&= \frac{\mathrm{e}^{- \frac{\lambda}{\theta}}}\theta
\int_0^1 
x^{\frac{s}\theta -1}
\mathrm{e}^{ \frac{\lambda}{\theta}x}
\,\mathrm{d}x
\\ \notag
&= \frac1\theta\int_0^1 
(1-x)^{\frac{s}{\theta}-1}
\mathrm{e}^{- \frac{\lambda}{\theta}x}
\,\mathrm{d}x\,.
\end{align}
Then $\E\bigl[\mathrm{e}^{-s \tau^A}\bigr]  \le \E\bigl[\mathrm{e}^{-s \tau^{\theta,A}}\bigr] 
= \frac\lambda{\lambda+s}\E\bigl[\mathrm{e}^{-s \beta^{\theta,A}}\bigr] <\infty$
for $0\le -s <\min(\lambda, \theta)$ and 
the Laplace transform $\E\bigl[\mathrm{e}^{-s \tau^A}\bigr]$ converges in  $\{\Re(s)> -\min(\lambda, \theta)\}$. 
\end{thm}
\begin{proof}
Theorem~\ref{thm-sto-dom-cl} implies that $L$ is stochastically dominated by the exponential random variable 
$\mathcal{E}(\theta)$, and the stochastic domination of $\tau^A$ by $\tau^{\theta,A}$ is obtained by coupling the
service durations. The rest uses again Takács~\cite[(37)--(39)]{Takacs1956} or \cite[pp.210,211]{Takacs1962}, 
\[
\int_0^\infty \mathrm{e}^{-st - \lambda\int_0^t(1-G^\theta(u))\,\mathrm{d} u}\,\mathrm{d}t
=
\int_0^\infty \mathrm{e}^{-st - \frac{\lambda}{\theta} (1-\mathrm{e}^{-\theta t})}\,\mathrm{d}t
\\
= \frac{\mathrm{e}^{- \frac{\lambda}{\theta}}}\theta
\int_0^1 
x^{\frac{s}\theta -1}
\mathrm{e}^{ \frac{\lambda}{\theta}x}
\,\mathrm{d}x
\]
and the change of variables $x\mapsto1-x$, as well as
analytic continuation (Rudin~\cite[Thm 10.18 p.208]{Rudin1987}) and Laplace transform properties (Widder~\cite[Thm~5b p.58]{Widder1941}).
\end{proof}

There are many expressions for  
\[
\int_0^\infty \mathrm{e}^{-st - \lambda\int_0^t(1-G^\theta(u))\,\mathrm{d} u}\,\mathrm{d}t\,,
\quad
\int_0^\infty \mathrm{e}^{-st - \lambda\int_0^t(1-G^{\theta,A}(u))\,\mathrm{d} u}\,\mathrm{d}t\,,
\]
and the Laplace transforms of $\tau^\theta \triangleq\tau^{\theta,0}$  and $\tau^{\theta,A}$ 
and of the corresponding busy periods $\beta^\theta \triangleq\beta^{\theta,0}$ and $\beta^{\theta,A}$. 
These result from \eqref{cdf-A}-\eqref{Atozero1}-\eqref{Atozero2}
with $F^A$ and $F$ replaced by $G^{\theta,A}$ and $G^\theta$
and the extensive study of the $M/M/\infty$ queue, which is the dominating queue in the case $A=0$, notably
by Guillemin and Simonian~\cite{Guillemin-Simonian1995} in relation to 
Kummer's confluent hypergeometric function. 
Appendix~\ref{appendix-queue}  lists some material on this topic which would be useful in actual applications. 


\section{Applications}
\label{sec-applications}

These regenerative proprieties enable to adapt classic and less classic limit theorems for i.i.d.\ random variables 
by decomposing the process into a delay and independent i.i.d.\ cycles, and to use coupling and other tools provided 
by regeneration such as found in the reference book Thorisson~\cite{Thorisson2000}.
For instance we can now provide a proof of Br\'emaud and Massouli\'e~\cite[Thm~1(c)]{bremaudmassoulie} using a 
general result for regenerative processes.   

\begin{thm}[Br\'emaud and Massouli\'e]
\label{thm-erg-law}
Assume that~\eqref{sub-crit-ass} and~\eqref{init-cond-ass} hold. 
Then there exists a stationary version $N^*$ of the Hawkes process on $(0,\infty)$ such that
\[
(S_t N)|_{(0,\infty)}
\xrightarrow[t\to\infty]{\textnormal{total variation}} N^*\,.
\]
\end{thm}
\begin{proof}
The Hawkes process is regenerative, the simple fact that $\tau^A$ is spread out is left as
an exercise (see~\cite[Sect.~4, Proof of Theorem~1.3~b)]{costa2018renewal} for a solution)
and the convergence statement follows from Thorisson~\cite[Thm~10.3.3 p.351]{Thorisson2000}.
\end{proof}

Good controls on $\tau^A_0$ are important in practice. We say that a point process $N'$ is stochastically dominated 
by a point process $N''$ if $N'(B)$ is stochastically dominated by $N''(B)$ for all Borel sets $B$. 
This is equivalent to the existence of a coupling such that $N' \le N''$.

\begin{thm}
\label{thm-delay}
Assume that~\eqref{sub-crit-ass} and~\eqref{first-mom-ass} hold.
Let $0\le A <\infty$. Assume that $N^\mathrm{in}$ is stochastically dominated by a stationary version of the Hawkes process.
Then $\tau^A_0$ is stochastically dominated by $U[\tau^A]^*$, where $U$ is uniform on $[0,1]$ and $[\tau^A]^*$ 
is an independent length-biased version of $\tau^A$, and for all  nonnegative Borel $f$,
\[
\E[f(U[\tau^A]^*)] = \frac1{\E[\tau^A]}\E\biggl[ \int_0^{ \tau^A}f(u)\,du\biggr]\,,
\quad
\E[\tau^A_0] \le  \E[U[\tau^A]^*]  = \frac{\E[(\tau^A)^2]}{2 \E[\tau^A]}\,.
\]
Thus  $\tau^A_0$ is a.s.\ finite, and $\E[\tau^A_0]<\infty$ if and only if  $\int t^2 h(t)\,\mathrm{d}t  <\infty$ 
and then has a finite bound expressed from Theorem~\ref{A-laplace}.
If moreover there exists $\theta >0$ such that $\int \mathrm{e}^{\theta t}h(t)\,\mathrm{d}t\le 1$
then $\tau^A_0$ is stochastically dominated by $U[\tau^{\theta,A}]^*$,
where $U$ is uniform on $[0,1]$ and $[\tau^{\theta,A}]^*$ is an independent length-biased version of 
$\tau^{\theta,A}$, and for all  nonnegative Borel $f$,
\[
\E[f(U[\tau^{\theta,A}]^*] = \frac1{\E[\tau^{\theta,A}]}\E\biggl[ \int_0^{ \tau^{\theta,A}}f(u)\,du\biggr]\,,
\quad
\E[\tau^A_0] \le  \frac{\E[(\tau^{\theta,A})^2]}{2 \E[\tau^{\theta,A}]}
<\infty\,,
\]
with an explicit finite bound obtained using Theorem~\ref{A-laplace}
with $\tau^A$, $\E[L]$, and $F^A$ replaced by $\tau^{\theta,A}$, $1/\theta$, and $G^{\theta,A}$.
\end{thm}
\begin{proof}
Thorisson~\cite[Thm~10.2.1 p.341]{Thorisson2000} and uniqueness in law of the stationary 
version of the Hawkes process on $\R$ imply that if $N^\mathrm{in}$ is a stationary version of the 
Hawkes process on $(-\infty,0]$ then $\tau^A_0$ has same law as $U[\tau^A]^*$. 
A monotony argument on the $M/G/\infty$ queue yields that if an initial condition
is stochastically dominated by a second one then the corresponding $\tau^A_0$ are in the same stochastic order.
The result on $\E[\tau^A_0]<\infty$  follows from Theorem~\ref{thm-takacs}.
The statement about stochastic domination by $U[\tau^{\theta,A}]^*$ follows from another
monotony argument on the $M/G/\infty$ queue.
\end{proof}

Long-time limit results for functional statistical estimators such as~\eqref{eq-erg-av} are important for applications.
The following is classic in this context.
A Borel function on $\mathcal{N}((-A,0])$ is said to be locally bounded if it is uniformly bounded on
$\{\nu \in \mathcal{N}((-A,0]) : \nu((-A,0])\le n\}$ for each $n\ge1$.

\begin{thm}[Pointwise ergodic theorem]
\label{thm-point-erg}
Assume that~\eqref{sub-crit-ass} and~\eqref{init-cond-ass} hold. Let $0\le A <\infty$. 
If $f$ is a locally bounded Borel  function on $\mathcal{N}((-A,0])$
which is  nonnegative or $\pi^A $-integrable (see below), then
\[
\frac1{T}\int_0^T f(S_t (N|_{(t-A,t]}))\,\mathrm{d}t
\xrightarrow[T\to\infty]{\textnormal{a.s.}}
\pi^A f \triangleq
\frac1{\E[\tau^A]} \E\Biggl[
\int_0^{\tau^A} f(S_t (\psi(\Gamma)|_{(t-A,t]}))\,\mathrm{d}t
\Biggr].
\]
\end{thm}
\begin{proof}
This is a simple consequence of the renewal reward theorem.
See~\cite[Sect.~4, Proof of Theorem~1.3~a)]{costa2018renewal} with the notation 
$X_t \triangleq S_t (N|_{(t-A,t]})$ for the classic proof.
\end{proof}

The following central limit theorem is not as commonly found, and its proof is considerably more involved.

\begin{thm}[Central limit theorem]
Assume that~\eqref{sub-crit-ass} and~\eqref{init-cond-ass} hold. Let $0\le A <\infty$. 
If $f$ is a locally bounded Borel  function on $\mathcal{N}((-A,0])$  which is $\pi^A $-integrable and such that 
\begin{equation}
\label{sig2f}
\sigma^2(f)
\triangleq \frac1{\E[\tau^A]} 
\E\Biggl[\biggl(\int_0^{\tau^A} \big(f(S_t (N|_{(t-A,t]})) -\pi_Af\big)\,dt\biggr)^2 \Biggr] <\infty\,,
\end{equation}
then
\[
\sqrt{T} \biggl(\frac1{T}\int_0^T f(S_t (N|_{(t-A,t]}))\,\mathrm{d}t - \pi^A f \biggr)
 \xrightarrow[T\to\infty]{\textnormal{in law}} \mathcal{N}(0, \sigma^2(f))\,.
\]
\end{thm}
\begin{proof}
See~\cite[Sect.~4, Proof of Theorem~1.4]{costa2018renewal} with the notation $X_t \triangleq S_t (N|_{(t-A,t]})$.
\end{proof}

Many other long-time limit results for i.i.d.\ sequences can be adapted to the functional statistic estimators 
such as~\eqref{eq-erg-av},  notably to establish and quantify asymptotic and non-asymptotic
convergence rates and confidence intervals for the pointwise ergodic Theorem~\ref{thm-point-erg}.

The main goal and achievement of Costa~\emph{et al.}~\cite[Thm~1.5, Cor.~1.6]{costa2018renewal} is to provide 
non-asymptotic concentration inequalities for Theorem~\ref{thm-point-erg} in a set-up allowing for 
self-inhibition in which $h$ can take negative values. 
Regeneration is proved in a way requiring that $h$ has support bounded by $A$ 
in order to consider an auxiliary Markov process with values in  $\mathcal{N}((-A,0])$.
The proof of the main result~\cite[Thm~1.5]{costa2018renewal} uses regeneration techniques and
Bernstein's inequality~\cite[Cor.~2.10 p.25, (2.17)--(2.18)~p.24]{massart2007concentration}. 
It was quite intricate and required to finely control a number of terms, 
and applied Bernstein's inequality to diverse deviations above and below the mean
in order to get good bounds. Its corollary~\cite[Cor.~1.6]{costa2018renewal} provides a less precise but 
more tractable statement.

The proofs in~\cite[Sect.~4]{costa2018renewal} can be reproduced here~\emph{verbatim},
with the notation $X_t \triangleq S_t (N|_{(t-A,t]})$, for linear Hawkes processes with 
nonnegative transfer functions having possibly an unbounded support, 
and hence in a context with unbounded
memory. The elements allowing to state~\cite[Thm~1.5]{costa2018renewal} are quite long,
and we refer the interested reader directly to it.
We satisfy ourselves with a transcription of~\cite[Cor.~1.6]{costa2018renewal} showing how the 
tight bounds  we have obtained come into play. 
In order to better understand it, see Theorem~\ref{thm-sto-dom-reg} and its reference to 
Theorem~\ref{thm-takacs} 
and the discussion after 
Theorem~\ref{thm-sto-dom-cl} and notably \eqref{theta-star}.

\begin{thm}
\label{thm:dev_simple}
For simplicity let $N^\mathrm{in}=0$.
Assume that~\eqref{sub-crit-ass} and~\eqref{init-cond-ass} hold, and that 
\[
\theta^* \triangleq \sup\bigl\{\theta >0 : \int \mathrm{e}^{\theta t}h(t)\,\mathrm{d}t \le 1\bigr\}>0\,.
\]
Then  $\E[\tau^A] =  \frac1\lambda \mathrm{e}^{\lambda (\E[L]+A)} <\infty$.
Choose $0<\alpha <\min(\lambda,\theta^*)$, which is such that $\E\bigl[\mathrm{e}^{\alpha\tau^A}\bigr]<\infty$,
and $-\infty < a < b <\infty$. Let
 \[
 v\triangleq\frac{2(b-a)^2}{\alpha^2}\biggr\lfloor\frac{T}{\E[\tau^A]}\biggr\rfloor 
\E\bigl[\mathrm{e}^{\alpha\tau^A}\bigr]\mathrm{e}^{\alpha \E[\tau^A]}\,,
\quad
c \triangleq \frac{|b-a|}{\alpha}\,.
 \]
If $f$ is a Borel  function on $\mathcal{N}((-A,0])$ with values in $[a,b]$
then, for $\eps>0$,
\begin{align*}
&\P\biggl( 
\biggl|\frac1T \int_0^T  f(S_t (N|_{(t-A,t]}))\,\mathrm{d}t - \pi_A f \biggr| \ge \eps \biggr)
\\
&\quad\le 4 \exp\Biggl(
-\frac{\bigl( T\eps-|b-a| \E[\tau^A]\bigr)^2}{4 \left(2v+ c(T\eps-|b-a| \E[\tau^A] \right)}  
\Biggr)
\end{align*}
or equivalently, for any $\eta $ in $(0,1)$ there exists $\eps_\eta>0$ such that
\begin{align*}
&\eps_\eta\triangleq \frac{1}{T}\Bigl(|b-a| \E[\tau^A] 
- 2c\log(\eta/4)
+\sqrt{\,4c^2\log^2(\eta/4)-8 v \log(\eta/4)}\,\Bigr)\,,
\\
&\P\biggl(\biggl|\frac1T \int_0^T  f(S_t (N|_{(t-A,t]}))\,\mathrm{d}t - \pi_A f \biggr| \ge \eps_\eta \biggr) \leq \eta\,.
\end{align*}
Moreover $\E\bigl[\mathrm{e}^{\alpha\tau^A}\bigr]\le\E\bigl[\mathrm{e}^{\alpha\tau^{\theta,A}}\bigr] $
for  $\alpha < \theta <\min(\lambda,\theta^*)$ and the r.h.s.\ has explicit forms given in 
Theorem~\ref{thm-sto-dom-reg} and Appendix~\ref{appendix-queue}.
\end{thm}

\appendix
\section{Approximation of Jump-Type Estimators}
\label{appendix-estim}

Estimators of the form~\eqref{eq-erg-av} are absolutely continuous processes of $T$.
We now show how these estimators allow to well approximate an estimator of the form~\eqref{eq-erg-jumps},
which is a pure-jump process of $T$. Recall that by hypothesis $w=w\ind_{[-A',0]}$ for some $A'<A$ and let
\begin{equation*}
f_w : \mu \in \mathcal{N}((-A,0])\mapsto 
\int_{\{-A< y \le x\le 0\}} \frac{w(y-x)}{y-x+A}\, \mu(\mathrm{d}x)\mu(\mathrm{d}y)
\end{equation*}
which is well-defined as a finite sum.
Using \eqref{shift} and $x-t -(y-t) = x-y$ and the Fubini theorem for bounded functions yields
\begin{align*}
&\int_0^T f_w((S_t N)|_{(-A,0]}))\,\mathrm{d}t
\\
&\quad=
\int_0^T \int_{\{-A< y \le x\le 0\}} \frac{w(y-x)}{y-x+A}\, (S_t N)(\mathrm{d}x) (S_t N)(\mathrm{d}y)
\,\mathrm{d}t
\\
&\quad=
\int_0^T \int_{\{t-A< y \le x\le t\}} \frac{w(y-x)}{y-x+A}\, N(\mathrm{d}x) N(\mathrm{d}y)
\,\mathrm{d}t
 \\
&\quad=
 \int_{\{-A< y \le x\le T\}} \int_{\{x\le t<(y+A)\wedge T\}}\frac{w(y-x)}{y-x+A}\,\mathrm{d}t\, N(\mathrm{d}x) N(\mathrm{d}y)
 \\
&\quad=
 \int_{\{-A< y \le x\le T\}} ((y+A)\wedge T-x)\frac{w(y-x)}{y-x+A}\, N(\mathrm{d}x) N(\mathrm{d}y)
\end{align*}
and thus
\begin{align*}
&\int_0^T f_w((S_t N)|_{(-A,0]}))\,\mathrm{d}t
 \\
&\quad=
 \int_{\{-A< y \le x\le T,\, y\le T-A\}} w(y-x)\,N(\mathrm{d}x) N(\mathrm{d}y)
  \\
&\qquad+
 \int_{\{T-A< y \le x\le T\}} \frac{T-x}{y-x+A} w(y-x)\,N(\mathrm{d}x) N(\mathrm{d}y)\,.
\end{align*}
Hence
\begin{align}
\label{diff-est}
&\frac1{T}\int_{\{-A < x\le y \le T\}} w(x-y)\, N(\mathrm{d}x)N(\mathrm{d}y) 
-\frac1{T}\int_0^T f_w((S_t N)|_{(-A,0]}))\,\mathrm{d}t
\\ \notag
&\quad
=\frac1{T}\int_{\{T-A< y \le x\le T\}} \frac{y-T+A}{y-x+A} w(y-x)\,N(\mathrm{d}x) N(\mathrm{d}y)
\end{align}
which can be easily controlled, for instance if $\lVert w\rVert_\infty<\infty$ then
\begin{align*}
&\biggl|\frac1{T}\int_{\{T-A< y \le x\le T\}} \frac{y-T+A}{y-x+A} w(y-x)\,N(\mathrm{d}x) N(\mathrm{d}y)\biggr|
\\
&\quad\le \frac{\lVert w\rVert_\infty }{2T} [N((T-A,T])^2+N((T-A,T])]
\end{align*}
and if $w\ge 0$ then
\begin{multline*}
\frac1{T}\int_0^T f_w((S_t N)|_{(-A,0]}))\,\mathrm{d}t
\le
\frac1{T}\int_{\{-A < x\le y \le T\}} w(x-y)\, N(\mathrm{d}x)N(\mathrm{d}y) 
\\
\le
\frac1{T}\int_0^{T+A} f_w((S_t N)|_{(-A,0]}))\,\mathrm{d}t
\end{multline*}
which can be used for signed $w$ using $w=w^+-w^-$ and linearity.

There are still practical problems to solve. In order to control $f_w$ in terms of $w$ 
we use the fact that $A'<A$ and 
\[
\frac{|w(y-x)|}{y-x+A} = \frac{|w(y-x)|\ind_{[-A',0]}(y-x)}{y-x+A} \le \frac{|w(y-x)|\ind_{[-A',0]}(y-x)}{A-A'}\,.
\]
Since $A$ may be chosen as large as desired this yields nice bounds, for instance if $A \ge A'+1$ then  
\[
\frac{|w(y-x)|}{y-x+A} \le |w(y-x)|\,.
\]
The second point is that $f_w$ is not an a.s.\ bounded function of $\mu$, but one can control
the probability of having an abnormally large number of jumps of $N$ on any $(t-A,t]$ for $0\le t\le T$ and use this 
to obtain in terms of $\lVert w^+\rVert_\infty<\infty$ and $\lVert w^-\rVert_\infty<\infty$  
and well-chosen $-\infty < a <b<\infty$ that
\[
\frac1{T}\int_0^T f_w((S_t N)|_{(-A,0]}))\,\mathrm{d}t 
=\frac1{T}\int_0^T (a\vee f_w \wedge b)((S_t N)|_{(-A,0]}))\,\mathrm{d}t 
\]
with given probability arbitrarily close to $1$ and then use Theorem~\ref{thm:dev_simple} or like results.

\section{Renewal Times and Busy Periods}
\label{appendix-queue} 

In the sequel $0\le A < \infty$ and $s\in \C$ and $\Re(s) > 0$ except if stated otherwise.
The product form
\begin{align}
\label{A-laplace-prod}
\E\bigl[\mathrm{e}^{-s \tau^A}\bigr] 
&= \frac\lambda{\lambda + s} \E\bigl[\mathrm{e}^{-s \beta^A}\bigr]
\\ \notag
&=  \frac\lambda{\lambda + s} 
\biggl(\frac{\lambda + s}\lambda
 -
\frac1{\lambda} 
\biggl(\int_0^\infty \mathrm{e}^{-st - \lambda \int_0^t (1-F^A(u))\,\mathrm{d} u}\,\mathrm{d}t
\biggr)^{-1}\biggr)
\end{align}
yields that the Laplace transform of the busy period can be written
\begin{align}
\label{B-laplace}
\E\bigl[\mathrm{e}^{-s \beta^A}\bigr] 
&=
\frac{\lambda + s}\lambda
 -
\frac1{\lambda} 
\biggl(\int_0^\infty \mathrm{e}^{-st - \lambda \int_0^t (1-F^A(u))\,\mathrm{d} u}\,\mathrm{d}t
\biggr)^{-1}
\\ \notag
&=
\frac{ 
\int_0^\infty F^A(t)\mathrm{e}^{-st - \lambda \int_0^t (1-F^A(u)) \,\mathrm{d} u}\,\mathrm{d}t}%
{ 
\int_0^\infty \mathrm{e}^{-st - \lambda \int_0^t (1-F^A(u))\,\mathrm{d} u}\,\mathrm{d}t}
\end{align}
using integration by parts to obtain that (since $\Re(s) > 0$)
\[
\int_0^\infty s\mathrm{e}^{-st - \lambda \int_0^t (1-F^A(u))\,\mathrm{d} u}\,\mathrm{d}t
=
1-\lambda \int_0^\infty (1-F^A(t)) \mathrm{e}^{-st - \lambda \int_0^t (1-F^A(u))\,\mathrm{d} u}\,\mathrm{d}t\,.
\]
Moreover~\eqref{cdf-A}-\eqref{Atozero1}-\eqref{Atozero2} allow to concentrate on the case $A=0$,
and provide relations such as 
\begin{align}
\label{lapl-A-zero-tau}
&\E\bigl[\mathrm{e}^{-s \tau^A}\bigr] 
= 1- \mathrm{e}^{(\lambda + s)A}\Bigl( \mathrm{e}^{(\lambda + s)A} -1
+ \bigl(1-\E\bigl[\mathrm{e}^{-s \tau^0}\bigr]\bigr)^{-1}\Bigr)^{-1}\,,
\\
\label{lapl-A-zero-beta}
&\E\bigl[\mathrm{e}^{-s \beta^A}\bigr] 
\\ \notag
&=\frac{\lambda + s}\lambda -
\frac1{\lambda} 
\biggl(
\frac{1- \mathrm{e}^{-(\lambda + s)A}}{\lambda + s}
+  \mathrm{e}^{-(\lambda + s)A}
\Bigr(
\lambda + s - \lambda\E\bigl[\mathrm{e}^{-s \beta^0}\bigr]
\Bigl)^{-1}
\biggr)^{-1}\,,
\end{align}
and formulæ such as
\begin{align} 
\label{B-laplace-Atozero}
\E\bigl[\mathrm{e}^{-s \beta^A}\bigr] 
&=\frac{ 
(\lambda + s)\int_0^\infty F(t) \mathrm{e}^{-st - \lambda\int_0^t(1-F(u))\,\mathrm{d} u}\,\mathrm{d}t
}%
{ 
\mathrm{e}^{(\lambda + s)A}- 1
+ (\lambda + s)\int_0^\infty \mathrm{e}^{-st - \lambda\int_0^t(1-F(u))\,\mathrm{d} u}\,\mathrm{d}t
}
\\ \notag
&=1-\frac{\mathrm{e}^{(\lambda + s)A}- 1
}%
{ 
\mathrm{e}^{(\lambda + s)A}- 1
+ (\lambda + s)\int_0^\infty \mathrm{e}^{-st - \lambda\int_0^t(1-F(u))\,\mathrm{d} u}\,\mathrm{d}t
}\,.
\end{align}

The renewal time $\tau^{\theta,A}$ and busy period  $\beta^{\theta,A}$ satisfy the same in which $F^A$ and $F$ 
are replaced respectively by  $G^{\theta,A}$ and $G^\theta$, made quite explicit by
\begin{align*}
&\int_0^\infty \mathrm{e}^{-st - \lambda \int_0^t(1-G^{\theta,A}(u))\,\mathrm{d} u}\,\mathrm{d}t
\\
&\quad= 
\frac{1- \mathrm{e}^{-(\lambda + s)A}}{\lambda + s}
+  \mathrm{e}^{-(\lambda + s)A}
\int_0^\infty \mathrm{e}^{-st - \lambda\int_0^t(1-G^\theta(u))\,\mathrm{d} u}\,\mathrm{d}t
\end{align*}
and~\eqref{int-mminfty}.
The resulting functions of $s$ are analytic together with their inverses in the half-space $\{s\in\C : \Re(s) > -\theta\}$.
Similar computations on the last expression in~\eqref{B-laplace}, with $F^A$ replaced by $G^\theta$,
 yield the classic form
\[
\E\bigl[\mathrm{e}^{-s \beta^{\theta,0}}\bigr] 
=\frac{\int_0^1 
x^{\frac{s}\theta -1} (1-x)
\mathrm{e}^{ \frac{\lambda}{\theta}x}
\,\mathrm{d}x}{\int_0^1 
x^{\frac{s}\theta -1}
\mathrm{e}^{ \frac{\lambda}{\theta}x}
\,\mathrm{d}x}
=
\frac{\int_0^1 
(1-x)^{\frac{s}{\theta}-1} x
\mathrm{e}^{- \frac{\lambda}{\theta}x}
\,\mathrm{d}x}{\int_0^1 
(1-x)^{\frac{s}{\theta}-1}
\mathrm{e}^{- \frac{\lambda}{\theta}x}
\,\mathrm{d}x}
\]
under which this Laplace transform can be obtained using martingale techniques, 
see Robert~\cite[Sect.~6.3]{Robert2003}.

Guillemin and Simonian~\cite[Sect.~4]{Guillemin-Simonian1995} have thoroughly investigated this form in 
relation with Kummer's confluent hypergeometric function.
See Lebedev~\cite[Sect.~9.11]{Lebedev1965} for the integral representation of this function that 
allows to express~\eqref{int-mminfty} and $\E\bigl[\mathrm{e}^{-s \beta^{\theta,0}}\bigr]$
in terms of it,
 and~\cite[Sect.~9.9--9.10]{Lebedev1965} for some of its its properties.

\section*{Aknowledgments}
The author wishes to thank the Chaire Mod\'elisation Math\'ematique et Biodiversit\'e 
for many fine meetings on mathematical biology
and Bastien Mallein for interesting exchanges on branching random walk.

\bibliographystyle{imsart-number}
\bibliography{biblio-Hawkes}

\end{document}